\documentclass[12pt,a4paper,twoside]{article}
\usepackage{graphicx}
\usepackage{amsfonts,amsmath,amssymb,amscd}
\usepackage{bezier, graphics,graphicx}
\usepackage{epsfig, layout,latexsym}
\usepackage{theorem}
\usepackage{cite}
\usepackage{indentfirst}
\usepackage{array}

 \usepackage[T1,T2A]{fontenc}
\usepackage[utf8]{inputenc}
\usepackage[english]{babel}

\AtBeginDocument{} \AtBeginDocument{}

\theorembodyfont{\sl}
\newtheorem{theorem}{Theorem}[section]

\newtheorem{lemma}[theorem]{Lemma}
\newtheorem{corollary}[theorem]{Corollary}
\newtheorem{remark}[theorem]{Remark}
\newtheorem{example}[theorem]{Example}

\numberwithin{equation}{section} \numberwithin{theorem}{section}
\newcommand{\card}{\text{card}}

\renewcommand{\thetheorem}{\arabic{section}.\arabic{theorem}}
\renewcommand{\theproposition}{\arabic{section}.\arabic{proposition}}
\renewcommand{\thedefinition}{\arabic{section}.\arabic{definition}}
\renewcommand{\thecorollary}{\arabic{section}.\arabic{corollary}}
\renewcommand{\thelemma}{\arabic{section}.\arabic{lemma}}
\renewcommand{\theremark}{\arabic{section}.\arabic{remark}}
\renewcommand{\theexample}{\arabic{section}.\arabic{example}}
\renewcommand{\theequation}{\arabic{section}.\arabic{equation}}

\makeatletter
\renewenvironment{thebibliography}[1]
 {\section*{\centerline{\rm\textsc{Bibliography}}}%
 \@mkboth{\MakeUppercase\refname}{\MakeUppercase\refname}%
 \list{\@biblabel{\@arabic\c@enumiv}}%
 {\settowidth\labelwidth{\@biblabel{#1}}%
 \leftmargin\labelwidth
 \advance\leftmargin\labelsep
 \@openbib@code
  \usecounter{enumiv}%
  \let\p@enumiv\@empty
  \renewcommand\theenumiv{\@arabic\c@enumiv}}%
  \sloppy
  \clubpenalty4000
  \@clubpenalty \clubpenalty
  \widowpenalty4000%
  \sfcode`\.\@m
  \setlength{\itemsep}{-0.1cm}}
  {\def\@noitemerr
  {\@latex@warning{Empty 'thebibliography' environment}}%
 \endlist}
\renewcommand{\@biblabel}[1]{#1.}

\makeatother \hfuzz=0.5pt \tolerance=500

\begin{document}
\begin{center}
		\textbf{Regularized summation of Fourier–Laguerre series\\ in the weighted sup-norm}
\end{center}
\vspace*{10mm} \centerline{\textsc {S.\,G.\,Solodky $\!\!{}^{\dag}$, Y.\,A.\,Volynets $\!\!{}^{\ddag}$}}

\vspace*{5mm}
\centerline{$\!\!{}^{\dag}\!\!$ Institute of Mathematics, National Academy of Sciences of Ukraine, Kyiv}
\centerline{$\!\!{}^{\ddag}\!\!$ National University of Kyiv-Mohyla Academy, Kyiv, Ukraine}

\begin{abstract}
{\bf Анотація.} Розглядається задача відновлення функцій, заданих на півпрямій, за неточною вхідною інформацією. Для запропонованих регуляризуючих методів підсумовування, побудованих на основі поліномів Лаґерра, досліджуються їхні апроксимаційні властивості на класах Вінера. Встановлюються умови, за яких запропоновані методи є стійкими щодо малих збурень вхідних даних та оптимальними за порядком як за точністю у зваженій $\sup$-нормі, так і за кількістю використаних коефіцієнтів Фур’є—Лаґерра.

\vspace*{2mm}

We consider the problem of recovering functions defined on the half-line from inexact input information.
For the proposed regularizing summation methods based on Laguerre polynomials, we analyze their approximation properties
on functions from Wiener classes.
We establish conditions under which these methods are stable with respect to small perturbations of the input data and
order-optimal not only in accuracy in the weighted $\sup$-norm,
but also in the number of Fourier--Laguerre coefficients used.

\vspace*{3mm}

\vspace*{3mm}

\end{abstract}

\vphantom{AAA}
\\

\section{Introduction. Preliminaries}  \label{prelim}
This paper is devoted to the optimal recovery of functions on the half-line by regularizing summation methods
based on Laguerre polynomials when the input information is inexact.
This problem arises in many applications in which our results may be useful. In particular,
in the study of time-dependent processes and in signal analysis,
when one needs to recover, for example, the response of a system after an impulse, a waiting-time distribution in a queue, or
a relaxation curve, Laguerre sums provide a convenient representation on the entire half-line and account for decay
at infinity \cite{Abate}, \cite{Wang}. For heat conduction, diffusion, radial transport equations, and
other boundary-value problems on the half-line, Laguerre polynomials avoid truncating the domain at an artificially
chosen point and imposing an artificial condition there \cite{Shen}, \cite{LWL}. They are also useful for approximating decaying functions
when a uniform approximation on the half-line, rather than values only on a finite grid, is required \cite{OcTh}.

There is now an extensive literature on various aspects of function approximation
using Laguerre polynomials (see, for example, \cite{Erd}, \cite{Muck70_1}, \cite{St}, \cite{MaMi},\cite{Shen},
\cite{Wang},  \cite{OcTh}).
Nevertheless, to the best of our knowledge, optimal recovery for the Laguerre system on the half-line with inexact
input information has not been considered previously.
We propose below a class of regularizing summation methods that are optimal in terms of
information complexity;
that is, they attain the best order of accuracy while using the smallest, in order, number of
perturbed Fourier--Laguerre coefficients. Similar problems have previously been studied, including
by one of the authors of this paper, for systems of polynomials orthogonal on a finite interval, in particular,
for Jacobi polynomials and trigonometric polynomials (see, e.g.,
\cite{Sol_Stas_JC2020}, \cite{Sol_Stas_JC2024}, \cite{Sem_Sol_2024}, \cite{Sem_Sol_2026}, \cite{Kys}).
In this paper, the ideas and concepts of this approach are extended to the half-line and the Laguerre polynomial system.

We also emphasize two facts established below alongside the main results
that are of independent interest beyond the optimal recovery problem.
First, for the error of the partial Fourier--Laguerre sums on the class $W^{\mu}_{s}$
in the weighted $\sup$-norm, we establish the order-sharp two-sided estimate
$\varepsilon_N (W^{\mu}_{s}) \asymp N^{-\mu-1/s+1}$
(Corollary \ref{Cor_En}). This estimate is known in the trigonometric case
(see Remark \ref{rem_trig}), whereas its lower bound appears to be new in the Laguerre case: in the
trigonometric case, the proof relies on the Dirichlet kernel, while here the lower bound
for the block is obtained in Lemma \ref{Lemma BE}.
Second, the condition $\mu>1-1/s$, under which all results of the paper are stated,
is the exact threshold for finiteness of the problem: for $\mu<1-1/s$, the quantities under study
are infinite, whereas for $\mu=1-1/s$, they are bounded away from zero uniformly in $N$
(Remark \ref{rem_sharp}).

The article is organized as follows.
Section 1 provides some background information on the proposed concept that will be needed later.
In Section 2, for summation methods from the proposed class, their accuracy estimates
in the weighted sup-norm are established.
In Section 3, the problem statements
for optimizing numerical summation methods are given.
Section 4 presents the main results of the work with sharp (in order) estimates of the quantities under study.
In Section 5, numerical experiments are provided to confirm the effectiveness of the proposed methods.

For the further presentation of the material, we need the following notation and concepts.
Throughout the paper, let $\alpha\ge 0$ be real (not necessarily an integer) and let $t>0$.
On the half-line $(0,\infty)$, we use the weight $w(t)=t^\alpha e^{-t}$.

By $L_{2,w}=L_{2,w}(0, \infty)$ we mean a weighted Hilbert space of real-valued functions
$f(t)$ which are square-summable on $(0,\infty)$ with the weight $w$.
In $L_{2,w}(0, \infty)$ the inner product and norm are introduced in the standard way
$$
\langle f, g\rangle=\int_{0}^{\infty} t^{\alpha} e^{-t} f(t) g(t) \ d t .
$$
$$
\|f\|_{2}^2 := \int_{0}^\infty t^{\alpha} e^{-t} |f(t)|^2 \, d t .
$$
It is well known (see, for example, \cite[Ch.5]{Szego}) that the system of Laguerre polynomials $\{L_k^{(\alpha)}\}$
can be represented as
$$
L_k^{(\alpha)}(t)
= \sum_{m=0}^k \frac{\Gamma(k+\alpha+1)}{\Gamma(m+\alpha+1) (k-m)!} \, \frac{(-t)^m}{m!} , \quad k\in \mathbb{N}_0 ,
$$
where $\mathbb{N}_0=\{0\}\bigcup\mathbb{N}$.
Since
$$
\int_0^\infty \left(L_k^{(\alpha)}\right)^2 w dt = \Gamma(k+\alpha+1)/k! ,
$$
the system of Laguerre polynomials orthonormal with respect to the weight $w$ has the form
$$
\ell_k^{(\alpha)}(t) := \left(\frac{k!}{\Gamma(k+\alpha+1)}\right)^{1/2} L_k^{(\alpha)}(t) .
$$
Then the norm in $L_{2,w}$ can be written as
$$
\|f\|_{2}^2 =
\sum_{k=0}^{\infty}|\langle f, \ell_{k}^{(\alpha)} \rangle|^2 < \infty ,
$$
where $$ \langle f, \ell_{k}^{(\alpha)}\rangle=\int_{0}^{\infty} t^\alpha e^{-t} f (t) \ell_{k}^{(\alpha)} dt, \quad
k=0,1,2,\ldots,
$$
are Fourier-Laguerre coefficients of $f$.

We denote the Laguerre function by $\varphi_k^{(\alpha)}(t)$:
$$
\varphi_k^{(\alpha)}(t) = \ell_k^{(\alpha)}(t) \sqrt{w(t)} .
$$
By $C$ we will mean the space of real-valued functions continuous on $(0,\infty)$,
and for $g \in C$ we put $$ \|g\|_{C}:= \sup\limits_{t > 0} |g(t)| \in [0,+\infty] .$$
This quantity is used below only for weighted functions $g=h\sqrt{w}$, for which it is
finite.
Moreover, let
$\ell_p$, $1\leq p\leq\infty$, be the space of numerical sequences
$\mathbf{x}=\{x_{k}\}_{k\in\mathbb{N}_0}$, such that the corresponding
relation
$$
\|\mathbf{x}\|_{\ell_p}  := \left\{
\begin{array}{cl}
\bigg(\sum\limits_{k\in\mathbb{N}_0} |x_{k}|^p\bigg)^{\frac{1}{p}} < \infty ,
 \ & 1\leq p<\infty ,
\\\\
\sup\limits_{k\in\mathbb{N}_0}  |x_{k}| < \infty ,
  \ & p=\infty ,
\end{array}
\right.
$$
is finite.

We introduce the space of functions
$$
W_{s}^\mu =\{f\in L_{2,w}: \quad \|f\|_{s,\mu}^s :=\sum_{k=0}^{\infty} ({\underline{k}})^{s\mu}|\langle f,
\ell_{k}^{(\alpha)}\rangle|^s<\infty\},\quad 1\le s<\infty,
$$
$$
W_{\infty}^\mu =\{f\in L_{2,w}:
 \|f\|_{\infty,\mu} =
 \sup\limits_{k\in \mathbb{N}_0}\, (\max\{1,k\})^{\mu}\,
 |\langle f , \ell_{k}^{(\alpha)} \rangle | < \infty\},
$$
where $\mu>0$,\ $\underline{k}=\max\{1,k\}$,\ $k=0,1,2,\dots$.
In what follows, we use the same notation for a space and its unit ball:
$W_{s}^{\mu} = \{f\in W_{s}^{\mu}\!: \|f\|_{s,\mu} \leq 1\}$, which we call
a function class. Whether $W_{s}^{\mu}$ denotes the space or the class will be clear
from the context. These spaces/classes are commonly studied in approximation theory
and are known as weighted (generalized) Wiener classes (see \cite{Kolom2023}).

We represent a function $f\in W_{s}^{\mu}$ as
\begin{equation}\label{function}
	f(t) = \sum_{k=0}^{\infty}\langle f, \ell^{(\alpha)}_{k}\rangle \ell^{(\alpha)}_{k}(t).
\end{equation}
For every $f \in L_{2,w}$ the series in (\ref{function}) converges in the norm
of $L_{2,w}$, and (\ref{function}) is understood as an equality in $L_{2,w}$.
Moreover, for $\mu>1-1/s$ one has
$\sum_{k=0}^{\infty}\bigl|\langle f, \ell^{(\alpha)}_{k}\rangle\bigr|<\infty$
for every $f\in W^{\mu}_{s}$ (this estimate is contained in the proof of
Lemma~\ref{lemma_BoundT1} below), while
$\|\varphi_k^{(\alpha)}\|_C \le C_\alpha$ for all $k$ (see (\ref{varphi}) below).
Hence the series
$\sum_{k=0}^{\infty}\langle f, \ell^{(\alpha)}_{k}\rangle\, \varphi^{(\alpha)}_{k}(t)$
converges absolutely and uniformly on $[0,\infty)$, so that the function
$f\sqrt{w}$ has a continuous representative on $[0,\infty)$, with which it is
identified throughout the paper. In particular, the norm $\|f \sqrt{w}\|_{C}$
is well defined for $f\in W^{\mu}_{s}$, and the equality (\ref{function}) holds
at every point $t>0$.

We now assume that instead of $\langle f,   \ell^{(\alpha)}_{k}\rangle $ we are given perturbed values.
Specifically, suppose we are given a sequence of real numbers
$
\mathbf{f}^\delta = \left( \langle f^\delta,   \ell^{(\alpha)}_{k} \rangle \right)_{k \in \mathbb{N}_0}.
$
We define the error sequence ${\mathbf{\mbox{\boldmath$\xi$\unboldmath}}}= \left(\xi_k\right)_{k \in \mathbb{N}_0}$, where
$
\xi_k = \langle f - f^\delta,   \ell^{(\alpha)}_{k} \rangle,
$
and  assume that   the input error is measured in the $\ell_p$ norm,
i.e. for some $1 \leq p \leq \infty$, the following perturbation condition holds
\begin{equation} \label{perturbation2}
	\|{\mathbf{\mbox{\boldmath$\xi$\unboldmath}}}\|_{\ell_p} \leq \delta, \quad 0 < \delta < 1.
\end{equation}

We shall study regularizing properties of summation methods
\begin{equation} \label{ModVer}
S_N^\nu f^\delta (t) = \sum_{k=0}^{N}
\nu_k^N \langle f^\delta,   \ell^{(\alpha)}_{k}\rangle
\ell^{(\alpha)}_{k}(t)
\end{equation}
for certain triangular arrays $\nu=\{\nu_k^N\}$, $k=0,\ldots,N$, $N\in \mathbb{N}$.
Such summation methods are called $\theta$-methods, see, e.g. \cite{Mathe&Per}, and play a
stabilizing role in the direct (well-posed) problem of recovery.
The quality of the summation methods $S_N^\nu f^\delta (t)$ will depend on the
truncation level $N$ and on properties of $\nu=\{\nu_k^N\}$, $k=0,\ldots,N$,
more precisely we shall assume, that there is $C(\nu)>0$ and some $\theta>0$ such that
\begin{equation} \label{qual}
|1- \nu_k^N| \le C(\nu) \left(\frac{k}{N}\right)^\theta,\quad 0^\theta=0, \quad  k=0,1,\ldots,N .
\end{equation}
By $\mathcal{S}^\theta$ we denote the class of all summation methods (\ref{ModVer}) that satisfies (\ref{qual}).
We note in passing, that as
a consequence, the array $\nu$ is uniformly bounded.
Throughout the paper, $c$ denotes a positive constant that may depend on
$\alpha$, $\mu$, $s$, $\theta$ and on the constant $C(\nu)$ of (\ref{qual}),
but never on $N$, $\delta$, $f$, $f^\delta$ or on the error sequence
${\mathbf{\mbox{\boldmath$\xi$\unboldmath}}}$; its value may change from line
to line. Accordingly, all the estimates below are stated for a method of
$\mathcal{S}^\theta$ fixed in advance, the dependence on the method being
only through $C(\nu)$; they are not claimed to be uniform over the whole
class $\mathcal{S}^\theta$, on which $C(\nu)$ is unbounded.
The following examples
show that assumption (\ref{qual}) is rather natural.

\begin{example}   \label{Ex1} \rm
Here are some examples of methods from the class $\mathcal{S}^\theta$.
\begin{enumerate}
\item
The Fej\'{e}r method of summation, where $\nu_k^N:= 1-\frac{k}{N+1}$, $k\le N$, meets (\ref{qual})
with $\theta=1$.
\item
The Abel-Poisson method of summation, where $\nu_k^N:= e^{-k/N}$, $k\le N$, meets (\ref{qual})
with $\theta=1$.
\item
The Gauss-Weierstrass method of summation, where $\nu_k^N:= e^{-k^2/N^2}$, $k\le N$, meets (\ref{qual})
with $\theta=2$.
\item
The Zygmund method of summation, where $\nu_k^N:= 1-(\frac{k}{N})^\sigma$, $\sigma>0$, $k\le N$, meets (\ref{qual})
with $\theta=\sigma$.
\item
The Fourier method of summation, where $\nu_k^N:= 1$, $k\le N$, meets (\ref{qual})
for every $\theta>0$.
\item
The de la Vall\'{e}e Poussin method of summation. Let $N=2n$, then $\nu_k^{2n}:= 1$, $k\le n$,
$\nu_k^{2n}:= (2n-k)/n$, $n+1\le k\le 2n$.
The method meets (\ref{qual}) for every finite $\theta>0$ with $C(\nu)=2^\theta$.
\end{enumerate}
\end{example}

\section{Error estimate}  \label{EE}
Let us write the error of the methods (\ref{ModVer}) as
\begin{equation}\label{fullError}
f(t)-S_N^\nu f^\delta(t)= \left(f(t)-S_N f(t)\right)+
\left(S_N f(t)-S_N^{\nu}f(t)\right)+\left(S_N^{\nu} f(t)-S_N^{\nu} f^\delta(t)\right),
\end{equation}
where
\begin{equation}  \label{m_F}
S_N f(t) = \sum_{k=0}^{N} \langle f, \ell^{(\alpha)}_{k}\rangle
\ell^{(\alpha)}_{k}(t)
\end{equation}
- is the Fourier summation method.

Since the error of method (\ref{ModVer}) diverges in the standard $C$-metric, we estimate in $C$
the difference
$$
\left(f(t)-S_N^\nu f^\delta(t)\right)\sqrt{w(t)} = \left(f(t)-S_N f(t)\right) \sqrt{w(t)}
$$
\begin{equation}\label{fullErrorW}
+ \left(S_N f(t)-S_N^{\nu}f(t)\right) \sqrt{w(t)} + \left(S_N^{\nu} f(t)-S_N^{\nu} f^\delta(t)\right) \sqrt{w(t)} ,
\end{equation}
that is, we estimate the error of method (\ref{ModVer}) in the weighted $\sup$-norm.

The parameter $N$ in (\ref{fullError}) should be chosen depending on
$\delta$, $p$, $s$ and $\mu$ so as to minimize the error estimate for (\ref{ModVer}).

The notation $A\preceq B$ means $A\le c B$, $A\succeq B$ means $A\ge c B$, $A\asymp B$ means
$A\preceq B\preceq A$.
Moreover, the notation $A\ll B$ means $A=o(B)$, $A\gg B$ means $B=o(A)$.

An upper bound for the norm of the first difference
on the right-hand side of (\ref{fullErrorW}) is contained in the following statement.

\vskip 2mm

\begin{lemma}\label{lemma_BoundT1}
Let $f\in W^\mu_{s}$, $1\leq s \le \infty$, $\mu>1-1/s$.
Then, for method (\ref{m_F}), we have
$$
\|(f-S_N f)\sqrt{w}\|_{C}\leq c\|f\|_{s,\mu} N^{-\mu-1/s+1} .
$$
\end{lemma}

\textit{Proof.}
It is known (see, for example, \cite[formulas (2.2) and (2.5)]{Muck70_2}) that, for every $k$,
\begin{equation}  \label{varphi}
\|\varphi_k^{(\alpha)}\|_C \le C_\alpha ,
\end{equation}
where $C_\alpha$ depends on $\alpha$ and is independent of $k$.

First, let $1<s<\infty$.
Then using (\ref{varphi}), the H\"{o}lder inequality and the definition of
$W^\mu_{s}$, we obtain for $\mu>1-1/s$
$$
\|\left(f-S_N f\right) \sqrt{w}\|_C =
\|\sum_{k=N+1}^{\infty} \langle f, \ell^{(\alpha)}_{k}\rangle
\varphi^{(\alpha)}_{k}\|_C
\le \sup_{k>N} \|\varphi^{(\alpha)}_{k}\|_C \sum_{k=N+1}^{\infty} |\langle f, \ell^{(\alpha)}_{k}\rangle |
$$
$$
\le C_\alpha \left(\sum_{k=N+1}^{\infty} k^{\mu s}|\langle f, \ell^{(\alpha)}_{k}\rangle |^s\right)^{1/s}
\left(\sum_{k=N+1}^{\infty} k^{-\mu \frac{s}{s-1}}\right)^{(s-1)/s}
$$
$$
\le c \|f\|_{s,\mu} N^{-\mu-1/s+1} .
$$
In the case of $s=1$ we have
$$
\sum_{k>N} |\langle f, \ell^{(\alpha)}_{k}\rangle |
= \sum_{k>N} k^{-\mu}\, k^\mu\, |\langle f, \ell^{(\alpha)}_{k}\rangle |
\le N^{-\mu} \|f\|_{1,\mu} .
$$
For $s=\infty$ the tail of the series is estimated by an integral comparison
$\sum_{k>N} k^{-\mu} \le N^{1-\mu}/(\mu-1)$.\\
$\Box$

\vskip 2mm

The norm of the second difference
on the right-hand side of (\ref{fullErrorW}) is bounded in the following statement.

\begin{lemma}\label{lemma_BoundT2}
Let $f\in W^\mu_{s}$, $1\leq s \le \infty$.
Then, for any method from $\mathcal{S}^\theta$, $\theta>\mu+1/s-1$, we have
$$
\|(S_N f-S^{\nu}_N f)\sqrt{w}\|_{C}\leq c\|f\|_{s,\mu} N^{-\mu-1/s+1} .
$$
\end{lemma}

\textit{Proof.}
First, let $1<s<\infty$. Using (\ref{qual}), (\ref{varphi}), the H\"{o}lder inequality and the definition of
$W^\mu_{s}$, we obtain
$$
\|(S_N f-S^{\nu}_N f) \sqrt{w}\|_C \le
\sum_{k=0}^{N} \left|1-\nu_k^N\right| |\langle f, \ell^{(\alpha)}_{k}\rangle |
\|\varphi^{(\alpha)}_{k}\|_C
$$
$$
\le C(\nu)\, C_\alpha\, N^{-\theta}\, \sum_{k=1}^{N} \, k^\mu \,|\langle f, \ell^{(\alpha)}_{k}\rangle |
\, k^{-\mu+\theta}
$$
$$
\le C(\nu)\, C_\alpha\, N^{-\theta}\, \left(\sum_{k=1}^{N} k^{\mu s}|\langle f, \ell^{(\alpha)}_{k}\rangle |^s\right)^{1/s}
\left(\sum_{k=1}^{N} k^{(-\mu+\theta) \frac{s}{s-1}}\right)^{(s-1)/s}
$$
$$
\le c \|f\|_{s,\mu} N^{-\mu-1/s+1} .
$$
In the case of $s=1$ we use the equality $\sup_{1\le k\le N} k^{\theta-\mu} = N^{\theta-\mu}$.
For $s=\infty$ we have $\theta>\mu-1$, i.e. $\theta-\mu>-1$, and apply the relation
$\sum_{k=1}^{N} k^{\theta-\mu} \le c N^{\theta-\mu+1}$.\\
$\Box$

\vskip 2mm
The norm of the third difference
on the right-hand side of (\ref{fullErrorW}) is bounded in the following statement.

\begin{lemma}\label{lemma_BoundT3}
Let $f\in L_{2,w}$, $1\leq p\le \infty$.
Then, for any method from $\mathcal{S}^\theta$, $\theta>0$, we have
$$
\|(S^{\nu}_N f-S^{\nu}_N f^\delta)\sqrt{w}\|_{C}\leq c\, \delta\, N^{1-1/p} .
$$
\end{lemma}

\textit{Proof.}
Using (\ref{perturbation2}), (\ref{qual}), (\ref{varphi}),
we obtain
$$
\|(S^{\nu}_N f-S^{\nu}_N f^\delta) \sqrt{w}\|_C \le
\sum_{k=0}^{N} \left|\nu_k^N\right| \left|\langle f - f^\delta, \ell^{(\alpha)}_{k}\rangle\right|
\|\varphi^{(\alpha)}_{k}\|_C
$$
$$
\le (1+C(\nu)) C_\alpha\, \sum_{k=0}^{N} |\xi_k|,\quad  \xi_k=\langle f - f^\delta, \ell^{(\alpha)}_{k}\rangle.
$$
First, let $1<p<\infty$. Applying the H\"{o}lder inequality with $p'=p/(p-1)$, we obtain
$$
\sum_{k=0}^{N} |\xi_k|
\le \left(\sum_{k=0}^{N} |\xi_k|^p\right)^{1/p} (N+1)^{1/p'}
\le \|{\mathbf{\mbox{\boldmath$\xi$\unboldmath}}}\|_{\ell_p} (N+1)^{1-1/p} .
$$
For $p=1$ the conjugate exponent $p'$ is not defined, and the same bound is
immediate, since $1-1/p=0$:
$$
\sum_{k=0}^{N} |\xi_k| \le \|{\mathbf{\mbox{\boldmath$\xi$\unboldmath}}}\|_{\ell_1}
\le \delta = \delta\, (N+1)^{1-1/p} .
$$
For $p=\infty$ we have $1/p=0$ and, again directly,
$$
\sum_{k=0}^{N} |\xi_k| \le (N+1)\, \|{\mathbf{\mbox{\boldmath$\xi$\unboldmath}}}\|_{\ell_\infty}
\le \delta\, (N+1)^{1-1/p} .
$$
Since $(N+1)^{1-1/p} \le 2\, N^{1-1/p}$ for every $N\ge 1$ and every
$1\le p\le\infty$, in all three cases
$$
\sum_{k=0}^{N} |\xi_k| \le 2 \delta\, N^{1-1/p} .
$$
This proves the lemma.\\
$\Box$

\vskip 2mm

The combination of Lemmas \ref{lemma_BoundT1}, \ref{lemma_BoundT2} and \ref{lemma_BoundT3} gives
\begin{theorem} \label{Th_up}
Let $f\in W^\mu_{s}$, $\|f\|_{s,\mu}\le 1$, $1\leq s \le \infty$, $\mu>1-1/s$,
and let the condition (\ref{perturbation2}) be satisfied for some $1\le p\le \infty$.
Then, for any method $S^{\nu}_N$ from $\mathcal{S}^\theta$, $\theta>\mu+1/s-1$,
and $N\asymp \delta^{-\frac{1}{\mu-1/p+1/s}}$, we have
 $$
\|(f-S^{\nu}_N f^\delta) \sqrt{w}\|_{C} \leq c \delta^{\frac{\mu+1/s-1}{\mu-1/p+1/s}}
\asymp N^{-\mu-1/s+1}  .
$$
\end{theorem}
\textit{Proof.}  Taking into account  Lemmas \ref{lemma_BoundT1}, \ref{lemma_BoundT2} and \ref{lemma_BoundT3},
equation (\ref{fullError}) gives
$$
\|(f-S^{\nu}_N f^\delta) \sqrt{w}\|_{C} \leq \|(f-S_N f) \sqrt{w}\|_{C} +
\|(S_N f - S^{\nu}_N f) \sqrt{w}\|_{C} + \|(S^{\nu}_N f-S^{\nu}_N f^\delta) \sqrt{w}\|_{C}
$$
$$
\leq c\, N^{-\mu-1/s+1} + c\, \delta N^{1-1/p}
= c\, N \left(N^{-\mu-1/s} + \delta N^{-1/p}\right) .
$$
Substituting the rule $N\asymp \delta^{-\frac{1}{\mu-1/p+1/s}}$ into the relation above completely proves Theorem.\\
$\Box$
\vskip 2mm

\begin{corollary} \label{Cor1}
In the considered problem, any summation method (\ref{ModVer})  from $\mathcal{S}^\theta$,
$\theta>\mu+1/s-1$, achieves the accuracy
$$
 O\Big(\delta^{\frac{\mu+1/s-1}{\mu-1/p+1/s}}\Big)
$$
on the class $W^{\mu}_{s}$, $\mu>1-1/s$, and requires
$$
\card(\{0,1,\ldots,N\}) \asymp
N \asymp \delta^{-\frac{1}{\mu-1/p+1/s}}
$$
perturbed Fourier-Laguerre coefficients.
\end{corollary}
\vskip 2mm

\begin{remark} \label{prior}  \rm
Methods from $\mathcal{S}^\theta$ were previously studied for numerical summation in \cite{Mathe&Per},
\cite{Shar}, \cite{Sol&Shar}, \cite{PoPo}. In particular, for $s=p=2$, the order $\delta^{(\mu-1/2)/\mu}$ and the rule
$N\asymp\delta^{-1/\mu}$ were obtained in \cite{Mathe&Per} (Theorem 3.1 with $\beta=0$) for methods of degree
$\theta\ge\mu$. Here this condition is weakened to $\theta>\mu-1/2$, which extends the admissible range by $1/2$
(for the Fej\'{e}r method, $\mu<3/2$ instead of $\mu\le1$).
\end{remark}
\vskip 2mm

\begin{remark}   \label{rem_theta}   \rm
Theorem \ref{Th_up} shows that the parameter $\theta$ for methods from $\mathcal{S}^\theta$ determines
the greatest smoothness $\mu$ for which the accuracy
$O(\delta^{\frac{\mu+1/s-1}{\mu-1/p+1/s}})$ is guaranteed: the condition $\theta>\mu+1/s-1$ must hold.
Thus, for the methods in Example \ref{Ex1},
we have the following ranges of values of $\mu$ permitted by Theorem \ref{Th_up}:\\
-- the Fej\'{e}r and Abel--Poisson methods: $1-1/s<\mu<2-1/s$; \\
-- the Gauss--Weierstrass method: $1-1/s<\mu<3-1/s$; \\
-- the Zygmund method: $1-1/s<\mu<\sigma+1-1/s$; \\
-- the Fourier and de la Vall\'{e}e Poussin methods: the entire scale $\mu>1-1/s$. \\
The Fourier and de la Vall\'{e}e Poussin methods are the only methods in Example \ref{Ex1} for which Theorem \ref{Th_up} applies to
the entire smoothness scale $\mu>1-1/s$.
\end{remark}

\section{Optimal recovery and information complexity. Problem setting} \label{opt_Lq_alpha}

Our next goal is to show that methods (\ref{ModVer}) from $\mathcal{S}^\theta$ attain order-optimal estimates
both in accuracy and in information complexity.
For this purpose, we introduce below minimax quantities that characterize the optimality of
numerical summation methods on the class $W^{\mu}_{s}$.

Let $m$ be an arbitrary mapping from $L_{2,w}$ into $C$. By $\mathcal{M}$ we denote
the set of all such mappings.
The {\it optimal recovery error} on $W^{\mu}_{s}$ is defined as the quantity
$$
E_{\delta} \left( W^{\mu}_{s},  \mathcal{M}, C \right) =
\inf_{m \in \mathcal{M}} \
\sup_{\substack{f \in W_{s}^\mu}} \
\sup_{\substack{f^\delta \in L_{2,w} \\ \|f - f^\delta\|_{2} \leq \delta}} \
\left\| (f - m f^\delta) \sqrt{w} \right\|_{C} ,
$$
where the element $m f^\delta \in C$ serves as an approximate solution of the problem.
The quantity $E_{\delta}$ characterizes the best possible accuracy in recovering a function
$f\in W^{\mu}_{s}$ from an arbitrary $\delta$-perturbation $f^\delta \in L_{2,w}$.
The methods over which the $\inf$ is taken are not required to be linear, continuous, or stable.
Despite the fact that this quantity answers the question of the minimal approximation error, it ignores the question
of the form and amount of discrete information needed to achieve a given accuracy.

To address these questions, we introduce two further minimax quantities. They characterize
the optimality of numerical summation methods that use, as input information,
perturbed Fourier--Laguerre coefficients with errors measured in the $\ell_p$ norms.

Now, as in Section \ref{prelim}, assume that input error is measured in the $\ell_p$-norm.
In other words, we suppose that for some $1 \leq p \leq \infty$,
the perturbation condition (\ref{perturbation2}) holds.

An {\it information operator}
$ G : L_{2,w} \to \ell_2$
is defined as a mapping that assigns to each function $f^\delta \in L_{2,w}$ its sequence of Fourier–Laguerre coefficients:
\begin{equation} \label{FHc}
G(f^\delta) = \mathbf{f}^\delta .
\end{equation}
Such discrete information, expressed through Fourier coefficients as in \eqref{FHc}, is commonly referred
to as {\it Galerkin information} (see, e.g., \cite{Wer87}).

A {\it recovery operator} is defined as any mapping
$$
\psi : \ell_2 \to C,
$$
which assigns to a sequence $\mathbf{f}^\delta \in \ell_2$ an element $\psi \mathbf{f}^\delta \in C$.

The {\it approximation method} is any operator of the form
\begin{equation} \label{DM}
\psi G : L_{2,w} \to C,
\end{equation}
where $G$ is an information operator (as in \eqref{FHc}) and $\psi$ is a recovery operator.
The element $\psi G(f^\delta) \in C$ is interpreted as an approximation for  the
function $f \in W_{s}^\mu$.
We denote by $\Psi$ the set of all such operators of the form \eqref{DM}.
The only requirement for the methods in $\Psi$ is that they use {\it Galerkin information} (i.e., the Fourier–Laguerre coefficients)
as discrete data. The number of coefficients involved may, in general, be infinite.

The error of a method $\psi G$ on the class $W_{s}^\mu$ is given by
$$
e_{\delta}(W_{s}^\mu, \psi G, C, \ell_p) =
\sup_{\substack{f \in W_{s}^\mu}} \
\sup_{\substack{f^\delta \in L_{2,w}: \ \\ \|\mathbf{f^\delta} - \mathbf{f}\|_{\ell_p} \leq \delta}} \
\left\| (f - \psi G(f^\delta)) \sqrt{w} \right\|_{C}.
$$
We study the quantity
$$
\mathcal{E}_\delta (W_{s}^\mu, \Psi, C, \ell_p) =
\inf_{\psi G \in \Psi} e_{\delta}(W_{s}^\mu, \psi G, C, \ell_p),
$$
which represents the best possible accuracy achievable in approximating an arbitrary function
$f \in W_{s}^\mu$, using methods based only on perturbed Fourier–Laguerre coefficients measured in the $\ell_p$ norm.
Note, the amount of such input information used by methods from $\Psi$ can be arbitrarily large, even infinite.

We are also interested in the minimal amount of Galerkin information necessary to achieve a given level of accuracy.
To this end, we consider a finite set $\Omega \subset \mathbb{N}_0$, with $\card(\Omega) = N$,
where $\card(\Omega)$ denotes the number of points in $\Omega$.

We now introduce the operator
$$
G_\Omega : L_{2,w} \to \ell_2,
$$
which acts as follows: for each function $f \in L_{2,w}$, it produces a numerical sequence whose $k$-th component
(for $k \in \mathbb{N}_0$) is defined as
$
\langle f, \ell^{(\alpha)}_k \rangle \quad \text{if } k \in \Omega,
$
and is set to zero otherwise. That is, the operator retains only the Fourier–Laguerre coefficients indexed by the finite set
$\Omega$ and discards all others.

A {\it numerical summation algorithm} is then any operator of the form
$$
\psi G_\Omega : L_{2,w} \to C,
$$
where, as before, $ \psi : \ell_2 \to C$ is a recovery operator.
The set of all such algorithms is denoted by $\Psi_\Omega$.
Define
$$
\Psi_N := \bigcup_{\Omega : \,\,\card(\Omega) \le N} \Psi_\Omega,
$$
where the union is taken over all index sets $\Omega$ containing at most $N$ elements.
Clearly, $\Psi_N \subset \Psi$, since the algorithms in $\Psi_N$ are restricted to using no more than $N$ Fourier–Laguerre coefficients.
In contrast, the methods from $\Psi$ can deal with  any amount of Galerkin information.

The  {\it minimal  radius of Galerkin information } for the problem of numerical summation over the class $W^{\mu}_{s}$ is given by
$$
R_{N,\delta}  (W^{\mu}_{s}, \Psi_N, C, \ell_p) =
\inf_{ \psi G_\Omega \in \Psi_N} e_\delta(W^{\mu}_{s},  \psi G_\Omega, C, \ell_p) .
$$
This quantity characterizes the best possible accuracy that can be achieved by numerical summation of an arbitrary function $f \in W^{\mu}_{s}$, using at most $N$ values of its Fourier-Laguerre coefficients that are $\delta$-perturbed in the $\ell_p$ norm.

The inclusion $\Psi_N \subset \Psi \subset {\mathcal M}$ implies that, for every $N$ and $\delta>0$,
\begin{equation}   \label{low1}
\mathcal{E}_\delta (W_{s}^\mu, \Psi, C, \ell_p)
\le R_{N,\delta}  (W^{\mu}_{s}, \Psi_N, C, \ell_p) ,\quad 1\le p\le \infty,
\end{equation}
\begin{equation}   \label{low2}
E_{\delta} \left(W^{\mu}_{s},  \mathcal{M}, C \right)
\le \mathcal{E}_\delta (W_{s}^\mu, \Psi, C, \ell_2)
\le R_{N,\delta}  (W^{\mu}_{s}, \Psi_N, C, \ell_2) .
\end{equation}
In our analysis, the quantities $\mathcal{E}_\delta$ and $E_{\delta}$ are
auxiliary, and their values provide lower bounds for $R_{N,\delta}$.

Finally, we are also interested in finding a value $N_{\min}$ such that
$$
\min\limits_N R_{N,\delta}(W^{\mu}_{s},\Psi_N, C, \ell_p)
\asymp R_{N_{\min},\delta}(W^{\mu}_{s},\Psi_{N_{\min}}, C, \ell_p) .
$$
The quantity $N_{\min}$ gives the smallest number of perturbed Fourier--Laguerre coefficients
whose use guarantees the highest order of accuracy on the entire class $W^{\mu}_{s}$.
Thus, $N_{\min}$ can be regarded as the information complexity of numerical summation.

\section{Sharp bounds of optimal recovery and information complexity}
We will need the following auxiliary result.

\begin{lemma} \label{Lemma BE}
For every real $\alpha\ge 0$, every positive integer $N$, and every set $\Omega$
consisting of $N$ integer points in the interval $[N,3N]$, the function
$$
f(t)=\sum_{k\in \Omega}  \ell^{(\alpha)}_k(t)
$$
satisfies
$$
\|f \sqrt{w}\|_C \ge \bar{c}\, N,
$$
where $\bar{c}=\frac{0.7039}{\Gamma(\alpha+1)}\, 12^{-\alpha/2}$.
\end{lemma}

The proof of Lemma \ref{Lemma BE} is given in Appendix A.
\vskip 2mm

Our immediate goal is to find lower bounds for the minimax quantities from Section \ref{opt_Lq_alpha}.
First, we consider the case when, instead of the exact values of the Fourier-Laguerre coefficients
of a function $f$, their perturbations $\mathbf{f^\delta}$ are known,
for which the estimate (\ref{perturbation2}) holds.

\begin{theorem} \label{Th_low1}
Let $1\leq s\le \infty$, $\mu>1-1/s$, $1\leq p \leq \infty$.
Then, for any $0<\delta<1$ and $N$, we have
\begin{equation}   \label{low_est_1}
R_{N,\delta}(W^{\mu}_{s}, \Psi_N, C, \ell_p)
\geq {\cal E}_{\delta}(W^{\mu}_{s}, \Psi, C, \ell_p)
\geq c\, \delta^{\frac{\mu+1/s-1}{\mu-1/p+1/s}} .
\end{equation}
\end{theorem}

{\bf Proof.}
For an arbitrary $0<\delta<1$, we choose a real number $\tau>0$ so that
\begin{equation}  \label{N_1}
2^{-\mu}\, \tau^{-\mu+1/p-1/s} = \delta .
\end{equation}
Let us denote by $N_1$ the integer $N_1=\lceil\tau\rceil$ closest to $\tau$ from above, and construct an auxiliary function
\begin{equation}  \label{f_1}
f_1(t)  = 2^{-\mu} \,  N_1^{-\mu-1/s}\, \mathop{{\sum}}\limits_{k=N_1}^{2N_1-1}
  \ell^{(\alpha)}_k(t) .
\end{equation}
It is easy to see that for any $s<\infty$ $\|f_1\|_{s,\mu}\leq 1$.
If $s=\infty$ we obtain
$$
\|f_1\|_{\infty,\mu} = \max_{N_1\le k\le 2N_1-1} k^\mu \, 2^{-\mu}\, N_1^{-\mu} \le (2N_1)^\mu 2^{-\mu} N_1^{-\mu} =1 .
$$
Further, since
$$
\|\mathbf{f_1}\|_{\ell_p}^p
= 2^{-\mu p} \,  N_1^{-\mu p - p/s + 1} ,
$$
then by the choice of $N_1$, we have
\begin{equation}  \label{f_1=delta}
\|\mathbf{f_1}\|_{\ell_p}
= 2^{-\mu} \,  N_1^{-\mu +1/p - 1/s} \le \delta.
\end{equation}
Lemma \ref{Lemma BE} gives
$$
\left\|\left(\mathop{{\sum}}\limits_{k=N_1}^{2N_1-1}
  \ell^{(\alpha)}_k\right) \sqrt{w}\right\|_{C}
\geq \bar{c}\,  N_1
$$
and hence
\begin{equation}  \label{f_1_BE}
\|f_1 \sqrt{w}\|_C \geq 2^{-\mu}\, \bar{c}\, N_1^{-\mu-1/s+1} \geq
c\, \delta^{\frac{\mu+1/s-1}{\mu-1/p+1/s}} .
\end{equation}
Next, we will use a well-known scheme for obtaining lower bounds (see, for example, \cite{OsJC}).
Since $f_1, -f_1\in W^\mu_s$, equation (\ref{f_1=delta}) shows that
$\mathbf{0} =(0,\ldots,0,\ldots)\in \ell_p$ is a $\delta$-perturbation for $\mathbf{f_1}$ and
$-\mathbf{f_1}$ simultaneously.
Therefore, for any $\psi G\in \Psi$, we have
$$
2\|f_1 \sqrt{w}\|_{C}
\le \|(f_1 - \psi G(0)) \sqrt{w}\|_{C}
+ \|(-f_1 - \psi G(0)) \sqrt{w}\|_{C} .
$$
This implies
$$
\sup_{\genfrac{}{}{0pt}{}{\substack{\|f\|_{s,\mu}\leq 1},}
{\|\mathbf{f}\|_{\ell_p} \leq \delta} }
\| f \sqrt{w}\|_{C}
 \leq \sup_{\substack{\|f\|_{s,\mu}\leq 1}}
\ \sup_{\genfrac{}{}{0pt}{}{\substack{f^\delta\in L_{2,w}: \, \mathbf{f}^{\delta}=\mathbf{f}+\mathbf{\xi}}}
{\|\mathbf{\xi}\|_{\ell_p} \leq \delta} }
 \| (f - \psi G(f^{\delta})) \sqrt{w} \|_{C} .
$$
Since $\psi G\in \Psi$ is arbitrary, (\ref{f_1_BE}) gives
\begin{equation}  \label{E_delta}
{\cal E}_{\delta}(W^{\mu}_{s}, \Psi, C, \ell_p)
\geq \| f_1 \sqrt{w}\|_{C}
\geq c \, \delta^{\frac{\mu+1/s-1}{\mu-1/p+1/s}} .
\end{equation}
Substituting (\ref{E_delta}) into (\ref{low1}) proves the theorem.\\
$\Box$
\vskip 2mm

\begin{remark}  \rm
Theorem~\ref{Th_low1} is an analogue of Theorem~6.1 in \cite{Sem_Sol_2026}, where the same
quantities were evaluated for the Chebyshev system. There are three differences. First, that theorem assumes $1\le s<\infty$,
whereas the case $s=\infty$ is also included here. Second, for the Chebyshev system, the required lower
bound for a block follows immediately: $T_k(1)=\sqrt{2/\pi}$ for every $k$, so the entire block
is coherent at the single point $t=1$. For Laguerre functions, such a point exists only when
$\alpha=0$; the factor $t^{\alpha/2}$ removes it when $\alpha>0$. This
step is instead provided by Lemma~\ref{Lemma BE}, which holds for every $\alpha\ge 0$ and for
an arbitrary set $\Omega$ of $N$ integer points in the interval $[N,3N]$. Third,
the upper bounds rely on the uniform envelope
$\sup_k\|\ell^{(\alpha)}_k\sqrt{w}\|_C=O(1)$ on the half-line, which is trivial in the compact
case.
\end{remark}
\vskip 2mm

The construction from the proof of Theorem \ref{Th_low1} also yields a two-sided
estimate for the error of the Fourier method (\ref{m_F}) on the class $W^{\mu}_{s}$,
which is of independent interest; for the trigonometric case see Remark \ref{rem_trig} below,
and for the comparison with the Chebyshev system see the remark above.
Denote
$$
\varepsilon_N (W^{\mu}_{s}) := \sup_{\|f\|_{s,\mu}\le 1} \|(f-S_N f)\sqrt{w}\|_{C} .
$$

\begin{corollary} \label{Cor_En}
Let $\alpha\ge 0$, $1\leq s\le \infty$, $\mu>1-1/s$. Then, for any $N$, we have
$$
\varepsilon_N (W^{\mu}_{s}) \asymp N^{-\mu-1/s+1} .
$$
\end{corollary}

\textit{Proof.}
The upper bound is contained in Lemma \ref{lemma_BoundT1}.
To prove the lower bound, we consider the function $f_1$ defined by (\ref{f_1})
with $N_1=N+1$; the rule (\ref{N_1}) for choosing $N_1$ is not used here.
As established in the proof of Theorem \ref{Th_low1}, $\|f_1\|_{s,\mu}\leq 1$ and
$$
\|f_1 \sqrt{w}\|_C \geq 2^{-\mu}\, \bar{c}\, N_1^{-\mu-1/s+1} .
$$
Moreover, $\langle f_1, \ell^{(\alpha)}_{k}\rangle = 0$ for all $0\le k\le N$,
whence $S_N f_1 = 0$. Therefore
$$
\varepsilon_N (W^{\mu}_{s}) \geq \|(f_1 - S_N f_1)\sqrt{w}\|_{C} = \|f_1 \sqrt{w}\|_{C}
\geq 2^{-\mu}\, \bar{c}\, (N+1)^{-\mu-1/s+1} .
$$
$\Box$
\vskip 2mm

\begin{remark} \label{rem_trig} \rm
In the trigonometric case, the result of Corollary \ref{Cor_En} is known
(for $s=1$, see \cite{NNS}; for $1<s<\infty$, see \cite{Chen}).
\end{remark}
\vskip 2mm

\begin{remark} \label{rem_sharp} \rm
The assumption $\mu>1-1/s$ is not a technical limitation of the technique employed,
but the exact threshold at which the quantities under study are finite.
Indeed, let $M>N$ be arbitrary and let $f_1$ be the function (\ref{f_1}) with $N_1=M$.
Then, exactly as in the proof of Corollary \ref{Cor_En},
$$
\varepsilon_N (W^{\mu}_{s}) \geq \|f_1 \sqrt{w}\|_{C} \geq 2^{-\mu}\, \bar{c}\, M^{-\mu-1/s+1} .
$$
Letting $M\to\infty$, we conclude that for $\mu<1-1/s$ the error
$\varepsilon_N (W^{\mu}_{s})$ is infinite for every $N$, while for $\mu=1-1/s$ we
have $\varepsilon_N (W^{\mu}_{s}) \geq 2^{-\mu}\, \bar{c} > 0$ for every $N$, so
that the Fourier method does not converge on $W^{\mu}_{s}$.
The same is true for the minimax quantities: repeating the proof of
Theorem \ref{Th7.3} below with the set $\Lambda_M$, $M\geq N$, in place of
$\Lambda_N$, we obtain
$$
R_{N,\delta}(W^{\mu}_{s},\Psi_N, C, \ell_p) \geq c\, M^{-\mu-1/s+1}
\quad \mbox{for any } M\geq N .
$$
Hence for $\mu<1-1/s$ the quantity $R_{N,\delta}$ is infinite for all $N$ and
$\delta$, and for $\mu=1-1/s$ it is bounded away from zero uniformly in $N$ and
$\delta$. Thus, the condition $\mu>1-1/s$ describes exactly the range of
parameters in which the considered recovery problem has a finite answer
vanishing as $\delta\to 0$.
\end{remark}
\vskip 2mm

The following statement contains order-optimal estimates for the quantities ${\cal E}_{\delta}$,
$R_{N,\delta}$ in the weighted $\sup$-norm.

\begin{theorem} \label{Th_opt1}
Let $1\leq s\le \infty$, $\mu>1-1/s$, $1\leq p \leq \infty$.
Then, for any $0<\delta<1$, we have
$$
{\cal E}_{\delta}(W^{\mu}_{s}, {\Psi}, C,\ell_p)
\asymp \delta^{\frac{\mu+1/s-1}{\mu-1/p+1/s}} .
$$
Moreover, for any $0<\delta<1$ and $N \asymp \delta^{-\frac{1}{\mu-1/p+1/s}}$, we have
$$
R_{N,\delta}(W^{\mu}_{s},\Psi_N, C, \ell_p)
\asymp \delta^{\frac{\mu+1/s-1}{\mu-1/p+1/s}}
\asymp N^{-\mu-1/s+1} .
$$
The order-optimal bounds of ${\cal E}_{\delta}$, $R_{N,\delta}$
are attained by any method (\ref{ModVer}) from $\mathcal{S}^\theta$, $\theta>\mu+1/s-1$.
\end{theorem}

{\bf Proof.}
The upper bounds follow from Theorem \ref{Th_up} and the lower bounds are found in Theorem \ref{Th_low1}.\\
$\Box$
\vskip 2mm

\begin{remark} \label{2change1}   \rm
As follows from Theorem \ref{Th_opt1}, narrowing the set of methods from $\Psi$
to $\Psi_N$ does not affect the optimal order of accuracy $O(\delta^{\frac{\mu+1/s-1}{\mu-1/p+1/s}})$.
In other words, no summation method from $\Psi$ dealing with an arbitrary (possibly infinite)
amount of Galerkin information provides a higher order of accuracy than method (\ref{ModVer}).
\end{remark}
\vskip 2mm

\begin{remark} \label{norma_C}   \rm
In the theory of ill-posed problems (see, for example, \cite[Section 2.8]{LuP}), it is well known that
a problem is well-posed if and only if its
solution can be potentially approximated with an accuracy of order $O(\delta)$.
Therefore, as follows from Theorem \ref{Th_opt1},
the summation problem in the $C$-norm is well-posed for any $1\le s\le \infty$ and $p=1$.
The problem is also well-posed in any weaker metric than $C$.
At the same time, the problem is ill-posed in $C$-norm for any $1\le s\le \infty$ if $p>1$.
\end{remark}

\vskip 2mm

Further, we establish the smallest value of $N$ at which one can achieve
the sharp (in the power scale) estimates for the quantity $R_{N,\delta}$.

\begin{theorem} \label{Th7.3}
Let $1\leq s\le \infty$, $1\leq p \leq \infty$, $\mu>1-1/s$.
Then, for any $0<\delta<1$ and $N$, we have
\begin{equation}  \label{low_est_2}
R_{N,\delta}(W^{\mu}_{s},\Psi_N, C, \ell_p)
\ge c\, \max \left\{
\delta^{\frac{\mu+1/s-1}{\mu-1/p+1/s}},\,\,
N^{-\mu-1/s+1}\right\} .
\end{equation}
\end{theorem}

{\bf Proof.}
Let $\delta$ and $N$ be arbitrary.
Let us fix an arbitrarily chosen set $\hat{\Omega}$, $\card(\hat{\Omega})\le N$, of points
$k$, $0\le k < \infty$, of the coordinate axis and take a set
$$
\Lambda_N \subset \{N, N+1,\ldots,3N\} ,
$$
consisting of $N$ points of the axis such that
$$
\hat{\Omega} \cap \Lambda_N = \varnothing .
$$
Note that such a set $\Lambda_N$ will always exist.
Further, we construct the following auxiliary function
$$
f_2 (t) = 3^{-\mu}\, N^{-\mu-1/s}\, \sum_{k\in \Lambda_N} \ell^{(\alpha)}_k (t) .
$$
It is easy to see that $\|f_2\|_{s,\mu}\leq 1$.
By means of Lemma \ref{Lemma BE} we obtain
\begin{equation}   \label{f_2_BE}
\|f_2 \sqrt{w}\|_{C}
\geq c\,  N^{-\mu-1/s+1} .
\end{equation}
For any $\delta>0$, the function $f_2$ itself is an admissible $\delta$-perturbation,
i.e. $f_2^\delta (t) := f_2 (t)$.
Taking into account the relationship
$G_{\hat{\Omega}}(f_2)=G_{\hat{\Omega}}(-f_2)=\mathbf{0}$,
for any $\psi G_{\hat{\Omega}}\in\Psi_{\hat{\Omega}}$ we have
$\psi G_{\hat{\Omega}}(f_2)=\psi G_{\hat{\Omega}}(-f_2)$ and
$$
2 \|f_2 \sqrt{w}\|_{C}
=  \|\Big(-f_2  - \psi G_{\hat{\Omega}}(-f_2) - f_2 + \psi G_{\hat{\Omega}}(f_2)\Big) \sqrt{w}\|_{C}
$$
$$
\le \|\Big(- f_2 - \psi G_{\hat{\Omega}}(-f_2^\delta)\Big) \sqrt{w}\|_{C}
+ \|\Big(f_2 - \psi G_{\hat{\Omega}}(f_2^\delta)\Big) \sqrt{w}\|_{C}
$$
$$
\le 2\, \sup_{\substack{\|f\|_{s,\mu}\leq 1}}
\ \sup_{\genfrac{}{}{0pt}{}{\substack{f^\delta\in L_{2,w}: \, \mathbf{f}^{\delta}=\mathbf{f}+\mathbf{\xi}}}
{\|\mathbf{\xi}\|_{\ell_p} \leq \delta} }
\|\Big( f - \psi G_{\hat{\Omega}}(f^{\delta})\Big) \sqrt{w} \|_{C}
$$
$$
:= 2\, e_\delta (W_{s}^\mu, \psi G_{\hat{\Omega}}, C, \ell_p) .
$$
Then, using (\ref{f_2_BE}), we establish that for any $\delta>0$ and $N$, the following holds:
$$
e_\delta (W_{s}^\mu, \psi G_{\hat{\Omega}}, C, \ell_p)
\ge \|f_2 \sqrt{w}\|_{C}
\ge c\, N^{-\mu-1/s+1} .
$$
Since $\psi G_{\hat{\Omega}}$ and $\hat{\Omega}$, $\card(\hat{\Omega})\le N$, are arbitrary, it follows that
\begin{equation}  \label{low_est_N_q}
R_{N,\delta} (W^{\mu}_{s},\Psi_N, C, \ell_p)
\ge  c\, N^{-\mu-1/s+1} .
\end{equation}
Combining (\ref{low_est_1}) and (\ref{low_est_N_q}) proves the theorem.\\
$\Box$
\vskip 2mm

\begin{corollary} \label{Cor3}
From (\ref{low_est_2}) it follows that for any $N\ll \delta^{-1/(\mu-1/p+1/s)}$ the relation
$$
R_{N,\delta}(W^{\mu}_{s},\Psi_N, C, \ell_p)
\gg \delta^{\frac{\mu+1/s-1}{\mu-1/p+1/s}}
$$
is satisfied. This means that for such $N$ the optimal order of accuracy cannot be achieved.
On the other hand, for all $N\ge c\, \delta^{-1/(\mu-1/p+1/s)}$ the optimal order can be achieved.
\end{corollary}

Thus, we obtain
\begin{theorem} \label{ThNEW_q}
Let $1\leq s\le \infty$, $1\leq p \leq \infty$, $\mu>1-1/s$.
The most economical (in order) choice of $N$ gives
\begin{equation}  \label{optNp}
N_{\min}\asymp \delta^{-\frac{1}{\mu-1/p+1/s}} ,
\end{equation}
for which we have
$$
\min\limits_N R_{N,\delta}(W^{\mu}_{s},\Psi_N, C, \ell_p)
\asymp R_{N_{\min},\delta}(W^{\mu}_{s},\Psi_{N_{\min}}, C, \ell_p)
\asymp \delta^{\frac{\mu+1/s-1}{\mu-1/p+1/s}} .
$$
\end{theorem}

{\bf Proof.}
The theorem follows from Theorems \ref{Th_opt1} and \ref{Th7.3}.\\
$\Box$
\vskip 2mm

\begin{remark}  \label{rmNEW}   \rm
In the framework of  IBC theory, the quantity $N_{\min}$ is usually called information complexity.
The value (\ref{optNp}) describes the smallest (in order) amount of discrete information of the form (\ref{FHc})
necessary to achieve the highest possible accuracy $O(\delta^{\frac{\mu+1/s-1}{\mu-1/p+1/s}})$.
Note also that the rule (\ref{optNp}) coincides, in order, with the
discretization rule $N \asymp \delta^{-\frac{1}{\mu-1/p+1/s}}$ of Theorems
\ref{Th_up} and \ref{Th_opt1}: the summation methods (\ref{ModVer}), which
attain the optimal accuracy, use precisely (in order) the
information-complexity amount of the Fourier--Laguerre coefficients.
\end{remark}

\vskip 3mm
Now we consider separately the case where the error of the input data is measured in $L_{2,w}$.
In this case, it is possible to establish sharp (in order) estimates for
$E_\delta$ as well.

\begin{theorem} \label{Th7.4}
Let $1\leq s\le \infty$, $\mu>1-1/s$. Then, for any $0<\delta<1$,
we have
$$
E_{\delta}(W^{\mu}_{s}, {\cal M}, C)
\geq c \, \delta^{\frac{\mu+1/s-1}{\mu+1/s-1/2}} .
$$
\end{theorem}

{\bf Proof.}
We follow the scheme of the proof of Theorem \ref{Th_low1} with $p=2$; the only
difference is that the perturbation is now measured in the norm of $L_{2,w}$
rather than in $\ell_2$, and by Parseval's identity these two settings coincide.

For an arbitrary $0<\delta<1$, we choose a real number $\tau>0$ so that
\begin{equation}  \label{N_1_L2}
2^{-\mu}\, \tau^{-\mu+1/2-1/s} = \delta ,
\end{equation}
set $N_1=\lceil\tau\rceil$ and let $f_1$ be the function defined by (\ref{f_1})
with this value of $N_1$. As in the proof of Theorem \ref{Th_low1},
$\|f_1\|_{s,\mu}\le 1$. Since the system $\{\ell^{(\alpha)}_k\}$ is orthonormal
in $L_{2,w}$, Parseval's identity gives
$$
\|f_1\|_{2} = \|\mathbf{f_1}\|_{\ell_2}
= 2^{-\mu}\, N_1^{-\mu+1/2-1/s} \le \delta ,
$$
where the last inequality holds by the choice of $N_1$, because
$-\mu+1/2-1/s<0$ for $\mu>1-1/s$. Hence the zero function
$f^\delta=0\in L_{2,w}$ is a $\delta$-perturbation of $f_1$ and $-f_1$
simultaneously.

Now let $m\in\mathcal{M}$ be arbitrary. Since $f_1,\,-f_1\in W^{\mu}_{s}$, the
triangle inequality yields
$$
2\,\|f_1 \sqrt{w}\|_{C}
\le \|(f_1 - m(0))\sqrt{w}\|_{C} + \|(-f_1 - m(0))\sqrt{w}\|_{C}
$$
$$
\le 2 \sup_{f \in W_{s}^\mu}\
\sup_{\substack{f^\delta \in L_{2,w} \\ \|f - f^\delta\|_{2} \leq \delta}}
\|(f - m f^\delta)\sqrt{w}\|_{C} .
$$
We emphasize that this argument uses no structural properties of $m$: it is
valid for an arbitrary mapping from $L_{2,w}$ into $C$.
Taking the infimum over $m\in\mathcal{M}$ and applying Lemma \ref{Lemma BE} as
in (\ref{f_1_BE}), we obtain
$$
E_{\delta}(W^{\mu}_{s}, \mathcal{M}, C)
\ge \|f_1\sqrt{w}\|_{C}
\ge 2^{-\mu}\,\bar{c}\, N_1^{-\mu-1/s+1}
\geq c\, \delta^{\frac{\mu+1/s-1}{\mu+1/s-1/2}} ,
$$
because (\ref{N_1_L2}) gives $N_1 \asymp \delta^{-\frac{1}{\mu+1/s-1/2}}$.
In the case $s=\infty$ the argument goes through with the usual modifications
($1/s=0$ throughout).\\
$\Box$

\vskip 2mm

\begin{theorem} \label{Th7.5}
Let $1\leq s\le \infty$, $\mu>1-1/s$.
Then, for any $0<\delta<1$, we have
$$
E_{\delta}(W^{\mu}_{s}, {\mathcal M}, C)
\asymp {\cal E}_{\delta}(W^{\mu}_{s},\Psi, C, \ell_2)
\asymp \delta^{\frac{\mu+1/s-1}{\mu+1/s-1/2}} .
$$
Moreover, for any $0<\delta<1$ and $N \asymp \delta^{-\frac{1}{\mu+1/s-1/2}}$, we have
$$
R_{N,\delta}(W^{\mu}_{s},\Psi_N, C, \ell_2)
\asymp \delta^{\frac{\mu+1/s-1}{\mu+1/s-1/2}}
\asymp N^{-\mu-1/s+1} .
$$
The order-optimal bounds of $E_{\delta}$, ${\cal E}_{\delta}$, $R_{N,\delta}$
are attained by any method (\ref{ModVer}) from $\mathcal{S}^\theta$,
$\theta>\mu+1/s-1$.
\end{theorem}

{\bf Proof.}
The upper bounds
follow from Theorem \ref{Th_up}, the lower bounds are found in Theorem \ref{Th7.4}, and
the passage between $E_{\delta}$, ${\cal E}_{\delta}$ and $R_{N,\delta}$ is given by
(\ref{low2}).
$\Box$

\vskip 2mm

For the comparison with the Chebyshev system (Theorem~6.2 of \cite{Sem_Sol_2026})
see the remark following the proof of Theorem \ref{Th_low1}.

\vskip 2mm

\begin{remark} \label{2change2_alpha}   \rm
As follows from Theorem \ref{Th7.5}, narrowing the set of methods from ${\mathcal M}$
to $\Psi_N$ and $\Psi$ does not affect the optimal order of accuracy
$O(\delta^{\frac{\mu+1/s-1}{\mu+1/s-1/2}})$.
In other words, no mapping from $L_{2,w}$ into $C$ provides a higher order of accuracy
than methods (\ref{ModVer}) from $\mathcal{S}^\theta$,
$\theta>\mu+1/s-1$.
\end{remark}
\vskip 2mm

\section{Computational experiments}   \label{exp}

We now give numerical illustrations of our theoretical results.
To illustrate the performance of the proposed summation methods, we conduct computational
experiments for two functions with different smoothness.
The computations were performed on a computer with a 10-core Apple M4 processor
(4 performance and 6 efficiency cores) and 24 GB of memory, running
macOS 26.6.1, using Python 3.14.6 with NumPy 2.5.1; the figures were produced with
Matplotlib 3.11.0. A single script generates all tables and figures in this section,
with a total running time of $0.44$ s.

Throughout this section, $p=s=2$, the weight is $w(t)=t^{\alpha}e^{-t}$ with $\alpha=1/2$ and
$\alpha=1$, and the input error levels are
$\delta = 10^{-4}, 10^{-5}, 10^{-6}, 10^{-7}, 10^{-8}$.
The experiment follows the setting of Sections 1--3 directly.
We specify four conventions on which it is based.

{\it Normalization.} All estimates in Sections 2 and 4 are stated on the unit ball of
the space $W^{\mu}_{s}$. Therefore, each function is normalized before the data are
perturbed, $\widehat{f} := f/\|f\|_{2,\mu}$, and every number reported below
refers to $\widehat{f}$. In particular, $\mathbf{f}$ is the vector of exact
Fourier--Laguerre coefficients of $\widehat{f}$, truncated at $k=N$.

{\it Perturbed data.} In our experiments, the perturbed data are defined by
$\mathbf{f}^\delta = \mathbf{f} + {\mathbf{\mbox{\boldmath$\xi$\unboldmath}}}$,
where the error sequence is a single random realization on the sphere of
radius $\delta$,
$$
{\mathbf{\mbox{\boldmath$\xi$\unboldmath}}}
 = \delta \, \frac{\mathbf{g}}{\|\mathbf{g}\|_{\ell_2}} , \qquad
\mathbf{g} = (g_0, \ldots, g_N) ,
$$
with independent standard normal variables $g_k$. Then
$\|{\mathbf{\mbox{\boldmath$\xi$\unboldmath}}}\|_{\ell_2} = \delta$ exactly, so the
perturbation condition (\ref{perturbation2}) for $p=2$ holds with equality.
The tables below report the error for one admissible perturbation of the full admissible
magnitude, rather than an average over random realizations. The realization was generated by
NumPy's PCG64 generator with seed $20260815$; for each error level, the
corresponding table row and figure use the same realization.

{\it Polynomial degree.} For $p=s=2$, Theorem \ref{Th_up} prescribes
$N \asymp \delta^{-1/\mu}$, and we take $N = \lceil \delta^{-1/\mu} \rceil$,
rounded up to the nearest even integer. Even $N$ is required for the de la Vall\'{e}e Poussin
means, for which $N=2n$, and is harmless for the Fourier sums.

{\it Error.} The reported error is the weighted uniform norm
$$
\max_{t>0} \bigl| (\widehat{f} - S^{\nu}_{N}\widehat{f}^\delta)(t)\sqrt{w(t)} \bigr| ,
$$
that is, the quantity estimated in Theorem \ref{Th_up}, rather than the
$L_{2,w}$ norm. The maximum is taken over a fixed grid of $10930$ points in the interval
$(0, 1964]$. Its right endpoint lies beyond the largest zero of every retained
$\ell^{(\alpha)}_{k}\sqrt{w}$; the contribution from $t$ beyond this endpoint is less than $10^{-6}$ of
the reported value.

\subsection{Example 1}

In this example, we consider the function $g_1(t) = e^{-3t}$.
Its Fourier--Laguerre coefficients $c_k=\langle g_1,\ell^{(\alpha)}_k\rangle$ decay
geometrically, and the ratio of consecutive coefficients
tends to $3/4$. Hence $\underline{k}^{\mu}c_k \to 0$ and
$\|g_1\|_{2,\mu} < \infty$ for {\it every} $\mu > 0$. Thus, this function has no
maximal smoothness in the scale $W^{\mu}_{2}$, and here $\mu$ is prescribed rather than
computed. We take $\mu = 3$ for both values of $\alpha$. With this choice of $\mu$, the peak of the
normalized $g_1\sqrt{w}$ is $54$--$92$ times higher than the largest
error level $\delta = 10^{-4}$, so Fig.~\ref{Fig1} displays the approximations
rather than the noise.

Table \ref{tbl1} presents the results of approximating $g_1$ by the Fourier and
de la Vall\'{e}e Poussin methods for $\alpha=1/2, 1$. Here $N$ denotes the degree of the
Laguerre polynomials used.

\begin{table}[ht]
\caption{Approximation of $g_1(t) = e^{-3t}$, $\mu = 3$. The normalization norms are
$\|g_1\|_{2,\mu} = 4.37\cdot 10^{1}$ for $\alpha = 1/2$ and
$4.21\cdot 10^{1}$ for $\alpha = 1$.}
\label{tbl1}
\centering
\begin{tabular}{c|c|cc|cc}
\hline
 & & \multicolumn{2}{c|}{$\alpha = 1/2$} & \multicolumn{2}{c}{$\alpha = 1$} \\
$\delta$ & $N$ & Fourier & de la Vall\'{e}e P. & Fourier & de la Vall\'{e}e P. \\
\hline
$10^{-4}$ &  22 & $9.65\cdot 10^{-5}$ & $1.32\cdot 10^{-4}$ & $7.14\cdot 10^{-5}$ & $9.90\cdot 10^{-5}$ \\
$10^{-5}$ &  48 & $1.23\cdot 10^{-5}$ & $6.20\cdot 10^{-6}$ & $8.29\cdot 10^{-6}$ & $9.18\cdot 10^{-6}$ \\
$10^{-6}$ & 100 & $5.21\cdot 10^{-7}$ & $2.96\cdot 10^{-7}$ & $7.60\cdot 10^{-7}$ & $6.41\cdot 10^{-7}$ \\
$10^{-7}$ & 216 & $5.50\cdot 10^{-8}$ & $5.41\cdot 10^{-8}$ & $5.08\cdot 10^{-8}$ & $4.38\cdot 10^{-8}$ \\
$10^{-8}$ & 466 & $8.91\cdot 10^{-9}$ & $7.31\cdot 10^{-9}$ & $3.10\cdot 10^{-9}$ & $3.12\cdot 10^{-9}$ \\
\hline
\end{tabular}
\end{table}

In every row of Table \ref{tbl1}, the computed error is smaller than the bound
$\delta^{(\mu-1/2)/\mu}$ from Theorem \ref{Th_up} for $p=s=2$, with the constant set
equal to one, by a factor ranging from $3.5$ to $70$.

Figure \ref{Fig1} shows the exact function $g_1$ and its two approximations
obtained by the two summation methods.

\begin{figure}[ht]
\centering
\includegraphics[width=0.92\textwidth]{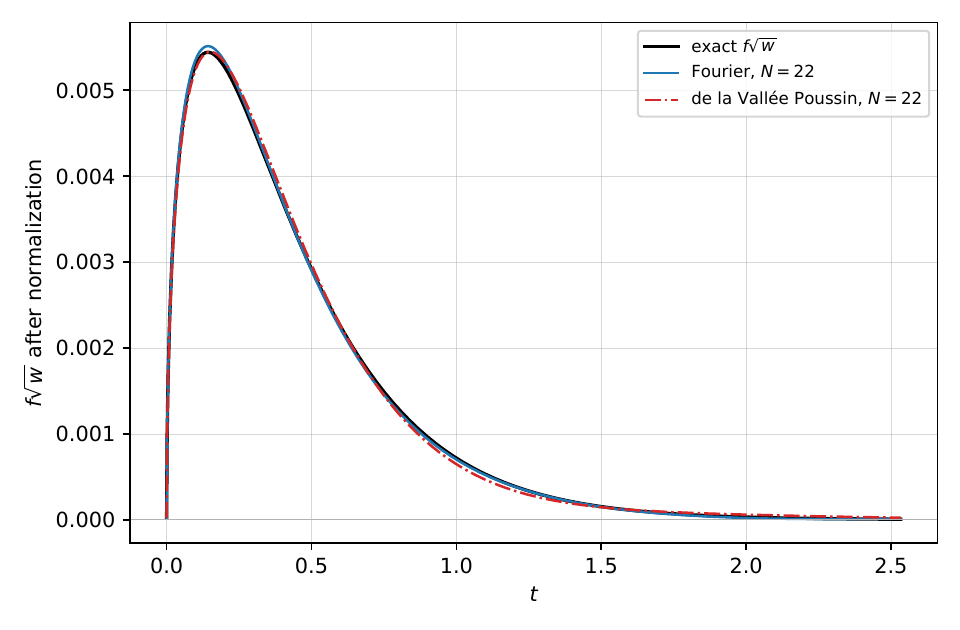}
\caption{Example 1, $\alpha = 1$, $\mu = 3$, $\delta = 10^{-4}$, $N = 22$: the exact
$g_1\sqrt{w}$ after normalization and its two approximations for the noise realization
corresponding to the first row of Table \ref{tbl1}.}
\label{Fig1}
\end{figure}

\subsection{Example 2}

We now consider the function $g_2(t) = t^b$ with $b=4.5$.
Its Fourier--Laguerre coefficients decay as $k^{-b-\alpha/2-1}$, so the terms
$\underline{k}^{2\mu}c_k^2$ in the norm have the exact order $k^{2\mu-2b-\alpha-2}$, and
the series converges if and only if
$\mu < \mu_{\max} := b + \alpha/2 + 1/2$; at the endpoint, it diverges
logarithmically. As the working value of $\mu$, we take the largest value with one decimal place that is strictly below
$\mu_{\max}$, namely, $\mu = 5.2$ for $\alpha = 1/2$ ($\mu_{\max} = 5.25$) and
$\mu = 5.4$ for $\alpha = 1$ ($\mu_{\max} = 5.5$).

Table \ref{tbl2} presents the results of approximating $g_2$ by the Fourier and
de la Vall\'{e}e Poussin methods for $\alpha=1/2, 1$.

\begin{table}[ht]
\caption{Approximation of $g_2(t) = t^{4.5}$. Left: $\alpha = 1/2$, $\mu = 5.2$,
$\|g_2\|_{2,\mu} = 3.32\cdot 10^{5}$. Right: $\alpha = 1$, $\mu = 5.4$,
$\|g_2\|_{2,\mu} = 6.93\cdot 10^{5}$.}
\label{tbl2}
\centering
\begin{tabular}{c|ccc|ccc}
\hline
 & \multicolumn{3}{c|}{$\alpha = 1/2$} & \multicolumn{3}{c}{$\alpha = 1$} \\
$\delta$ & $N$ & Fourier & de la Vall\'{e}e P. & $N$ & Fourier & de la Vall\'{e}e P. \\
\hline
$10^{-4}$ &  6 & $1.00\cdot 10^{-4}$ & $9.73\cdot 10^{-5}$ &  6 & $8.11\cdot 10^{-5}$ & $1.15\cdot 10^{-4}$ \\
$10^{-5}$ & 10 & $7.28\cdot 10^{-6}$ & $4.24\cdot 10^{-6}$ & 10 & $6.84\cdot 10^{-6}$ & $6.33\cdot 10^{-6}$ \\
$10^{-6}$ & 16 & $8.49\cdot 10^{-7}$ & $7.03\cdot 10^{-7}$ & 14 & $3.98\cdot 10^{-7}$ & $3.89\cdot 10^{-7}$ \\
$10^{-7}$ & 24 & $4.23\cdot 10^{-8}$ & $3.24\cdot 10^{-8}$ & 20 & $6.81\cdot 10^{-8}$ & $7.27\cdot 10^{-8}$ \\
$10^{-8}$ & 36 & $8.10\cdot 10^{-9}$ & $2.74\cdot 10^{-9}$ & 32 & $4.50\cdot 10^{-9}$ & $3.48\cdot 10^{-9}$ \\
\hline
\end{tabular}
\end{table}

In every row of Table \ref{tbl2}, the computed error is smaller than the same bound
$\delta^{(\mu-1/2)/\mu}$ by a factor ranging from $2.0$ to $21$.

Figure \ref{Fig2} shows the exact function $g_2$ and its two approximations
obtained by the two summation methods.

\begin{figure}[ht]
\centering
\includegraphics[width=0.92\textwidth]{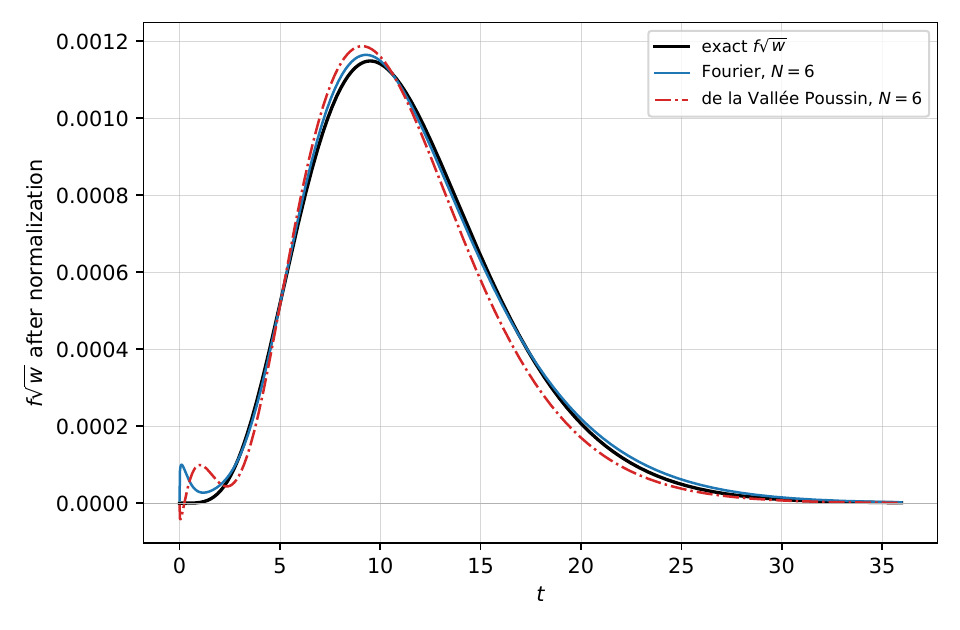}
\caption{Example 2, $\alpha = 1/2$, $\mu = 5.2$, $\delta = 10^{-4}$, $N = 6$:
the exact $g_2\sqrt{w}$ after normalization and its two approximations for the noise
realization corresponding to the first row of Table \ref{tbl2}.}
\label{Fig2}
\end{figure}

\appendix
\newpage

\renewcommand{\thetheorem}{A.\arabic{theorem}}
\renewcommand{\theproposition}{A.\arabic{proposition}}
\renewcommand{\thedefinition}{A.\arabic{definition}}
\renewcommand{\thecorollary}{A.\arabic{corollary}}
\renewcommand{\thelemma}{A.\arabic{lemma}}
\renewcommand{\theremark}{A.\arabic{remark}}
\renewcommand{\theexample}{A.\arabic{example}}
\renewcommand{\theequation}{A.\arabic{equation}}
\section{Proof of Lemma \ref{Lemma BE}}

{\bf Proof.}
A technical difficulty arises in obtaining a lower bound for the norm of $f \sqrt{w}$ in the uniform metric.
We describe it first.
A standard technique for obtaining lower bounds is to pass from the $C$-norm to the value of the function
at a point where it attains its maximum. This is possible when all terms in the sum attain their maxima
at the same point. Such a situation occurs, for example, for the Jacobi polynomials $p^{(\alpha,\beta)}_k$,
$\alpha \geq \beta \geq -1/2$, at the point $1$.
For Laguerre polynomials, however, the points of maximum depend on the index $k$ and therefore do not coincide.
We overcome this difficulty by finding a point common to all the polynomials, say $t_*$, at which all terms
are positive and have the same order as their maxima.
To choose $t_*$, consider the behavior of the Laguerre polynomials.
They are positive before their first zero and then oscillate with different periods, so that they
almost completely cancel one another.
It is therefore natural to choose the common point near $0$. We take $t_*=\frac{1}{12N}$.
We show that, at $t_*$, all terms of $f$ are positive.

Write
$$
L_k^{(\alpha)} (t) = \sum_{j=0}^k (-1)^j\, \binom{k+\alpha}{k-j}\, \frac{t^j}{j!}
= \sum_{j=0}^k (-1)^j a_j(t) ,
$$
where $a_j(t) = \binom{k+\alpha}{k-j}\, \frac{t^j}{j!}$.
Since
$$
\frac{\Gamma(\alpha+j+1)}{\Gamma(\alpha+1)} = (\alpha+1) \ldots (\alpha+j) \geq j! ,
$$
$$
\frac{k!}{(k-j)!} \leq k^j ,
$$
we have
$$
\frac{a_j(t)}{a_0(t)} = \frac{t^j}{j!}\, \frac{\Gamma(\alpha+1)\, k!}{\Gamma(\alpha+j+1)\, (k-j)!}
\le \frac{(kt)^j}{(j!)^2} .
$$
For $k\le 3N$ and $t=t_*$, we have $k t_*\le 1/4$, and hence
$$
\sum_{j=1}^k a_j(t_*) \le a_0\, \sum_{j\ge 1} \frac{4^{-j}}{(j!)^2}
= a_0 (I_0(1)-1) \le 0.2660664 a_0 ,
$$
where $I_0(x)$ is the modified Bessel function (see, for example, \cite[pp.~375 and 416]{AbrSt}):
$$
I_0(x)=\sum_{j\ge 0} \frac{(x/2)^{2j}}{(j!)^2} \quad \mbox{and}\quad I_0(1)=1.26606588\ldots .
$$
Since $a_0=L_k^{(\alpha)}(0)=\frac{\Gamma(k+\alpha+1)}{\Gamma(\alpha+1) k!}>0$, it follows that
$$
L_k^{(\alpha)}(t_*) \geq a_0 - \sum_{j\ge 1} a_j(t_*) \ge 0.7339336 L_k^{(\alpha)}(0) > 0 .
$$
We now derive the required estimate. Multiplication by the normalization factor gives
$$
\sqrt{\frac{k!}{\Gamma(k+\alpha+1)}}\, L_k^{(\alpha)}(0) = \frac{1}{\Gamma(\alpha+1)}\,
\sqrt{\frac{\Gamma(k+\alpha+1)}{k!}} .
$$
Next, let $\chi=(\ln \Gamma)'$ denote the digamma function (\cite[formula 6.3.1]{AbrSt}),
and let $\gamma$ denote Euler's constant:
$$
\gamma = \lim_{m\to \infty} \Big( \sum_{j=1}^m j^{-1} - \ln m \Big) = 0.5772156649\ldots
$$
(see \cite[formula 6.1.3]{AbrSt}). Clearly, $\chi$ is increasing.
By the inequality $\ln (1+1/k) > 1/(k+1)$, the sequence
$\sum_{j=1}^k j^{-1} - \ln k$ is decreasing and is therefore no smaller than its limit $\gamma$. Hence,
using \cite[formula 6.3.2]{AbrSt},
$$
\chi(k+1) = -\gamma + \sum_{j=1}^k j^{-1} \geq \ln k .
$$
Therefore,
$$
\ln \frac{\Gamma(k+\alpha+1)}{k!} = \int_0^\alpha \chi(k+1+u) du \geq \alpha \chi(k+1) \geq \alpha \ln k.
$$
This holds for every $\alpha \geq 0$ and $k\ge 1$. Thus,
$$
\frac{\Gamma(k+\alpha+1)}{k!} \ge k^\alpha \ge N^\alpha .
$$
Summing all terms in $f$ and using their positivity at $t_*$, we obtain
$$
|f(t_*)| \geq 0.7339336\, N\, \frac{N^{\alpha/2}}{\Gamma(\alpha+1)} .
$$
At the same point, the weight satisfies
$$
\sqrt{w(t_*)} = t_*^{\alpha/2} e^{-t_*/2}
= (12 N)^{-\alpha/2}\, e^{-1/(24N)} \ge (12 N)^{-\alpha/2}\, e^{-1/24} .
$$
Since $e^{-1/24}\, 0.7339336 \approx 0.7039813$, it follows that
$$
\|f \sqrt{w}\|_C \geq |f(t_*)|\, |\sqrt{w(t_*)}|
\geq \frac{0.7039}{\Gamma(\alpha+1)}\, 12^{-\alpha/2}\, N .
$$
This proves the lemma. \\
$\Box$

\medskip

\begin{remark}  \label{rem_A}  \rm
The choice $t_*=\frac{1}{12N}$ is a convenient point within the admissible
range rather than the optimum. The same calculation with $t_*=c/N$ gives $k t_*\le 3c$ for $k\le 3N$ and
$$
\sum_{j\ge1} a_j(t_*) \le a_0\bigl(I_0(2\sqrt{3c})-1\bigr),
$$
so Lemma~\ref{Lemma BE} holds with the constant
$$
\bar c(c)=\frac{\kappa(c)\, e^{-c/2}\, c^{\alpha/2}}{\Gamma(\alpha+1)},
\qquad
\kappa(c)=2-I_0\bigl(2\sqrt{3c}\bigr),
$$
for every $c$ such that $\kappa(c)>0$, that is, precisely when
$c<z_2^2/12=0.2723742\ldots$, where $z_2=1.8078967\ldots$ is the root of the equation $I_0(z)=2$.
The value $c=1/12$ corresponds to $\kappa=2-I_0(1)=0.733934\ldots$. The optimal $c$ depends only
on $\alpha$: for $\alpha=1$, it is $c=0.093$, which is nearly indistinguishable from $1/12$; for
$\alpha=5$, the optimal value $c=0.195$ increases $\bar c$ by a factor of $3.5$, and as $\alpha$
increases, the gain is unbounded.
\end{remark}

\end{document}